\documentclass{amsart}
\usepackage{latexsym}
\usepackage{amssymb}
\usepackage{amsmath}
\usepackage{amsfonts}
   
\def\b{\textbf{b}}

\def\R{\mathbb{R}}

\newcommand{\bmat}{\left[\begin{matrix}}
\newcommand{\emat}{\end{matrix}\right]}
\usepackage[latin1]{inputenc} 
\usepackage{amsthm}
\usepackage{amsxtra}
\usepackage{amscd}
\usepackage{setspace}   % enables line spacing commands 
\usepackage{mathrsfs}   % enables \mathscr

\newtheorem{theorem}{Theorem}
\newtheorem{proposition}[theorem]{Proposition}

\newtheorem{conjecture}[theorem]{Conjecture}

\theoremstyle{remark}
\newtheorem*{remark}{Remark}

\theoremstyle{definition}

\newcommand{\Z}{\mathbb{Z}}

\newcommand{\vb}{{\mathbf b}}

\usepackage{tikz}
\usetikzlibrary{arrows}

\usepackage{algorithm}
\usepackage{algpseudocode}
\usepackage{enumerate}
\usepackage{graphicx}

    \title{LLL and stochastic sandpile models}
    \author{Jintai Ding, Seungki Kim, Tsuyoshi Takagi, Yuntao Wang}
\begin{document}

\begin{abstract}

The aim of the present paper is to suggest that statistical physics provides the correct language to understand the practical behavior of the LLL algorithm, most of which are left unexplained to this day. To this end, we propose sandpile models that imitate LLL with compelling accuracy, and prove for these models some of the most desired statements regarding LLL. We also formulate a few conjectures that formally capture our heuristics and would serve as milestones for further development of the theory.

%We introduce stochastic sandpile models which imitate numerous aspects of the practical behavior of the LLL algorithm with compelling accuracy. In addition, we argue that the physics and mathematics of sandpile models provide satisfactory heuristic explanations to much of the mysteries of LLL, and pleasant implications for lattice-based cryptography as a whole. Based on these successes, we suggest a paradigm in which one regards blockwise reduction algorithms as 1-d stochastic self-organized criticality(SOC) models and study them as such.

\end{abstract}

\maketitle

\section{Introduction}

\subsection{The mysteries of LLL}

The LLL algorithm (\cite{LLL82}) is one of the most celebrated algorithmic inventions of the twentieth century, with countless applications to pure and computational number theory, computational science, and cryptography. It is also the most fundamental of lattice reduction algorithms, in that nearly all known reduction algorithms are generalizations of LLL in some sense, and they also utilize LLL as their subroutine. (We refer the reader to \cite{NV10} for a thorough survey on LLL and these related topics.) Thus it is rather curious that much of the observed behavior of LLL in practice is left totally unexplained, not even in a heuristic, speculative sense, even to this day.

The most well-known among the mysteries of LLL is the gap between its worst-case root Hermite factor(RHF) and the observed average-case, as documented in Nguyen and Stehl\'e (\cite{NS06}). It is a theorem from the original LLL paper (\cite{LLL82}) that the shortest vector of an LLL-reduced basis, with its determinant normalized to $1$, has length at most $(4/3)^{\frac{n-1}{4}} \approx 1.075^n$, whereas in practice one almost always observes $\approx 1.02^n$, regardless of the way in which the input is sampled. This is a strange phenomenon in the light of the work of Kim and Venkatesh (\cite{KV17}), which, roughly speaking, proves that, for almost every lattice, nearly all of its LLL bases have RHF close to the worst bound. This leads to the suspicion that the LLL algorithm must be operating in a complex manner that belies the simplicity of its code.

There are also many other phenomena regarding LLL that are unaccounted for. One is the geometric series assumption(GSA), originally proposed by Schnorr (\cite{S03}), and its partial failure at the boundaries, both of which are observed in other blockwise reduction algorithms as well e.g. BKZ (\cite{SE94}). There are also questions raised regarding the time complexity of LLL. Nguyen and Stehl\'e (\cite{NS06}) suggest that, in some situations, the average time complexity is lower than the worst-case, and in others, the worst-case is attained. The complexity of the optimal LLL algorithm --- i.e. the parameter $\delta$ equals $1$ --- is not proven to be polynomial-time, although observations suggest that it is (see Akhavi (\cite{A00}) and references therein).

\subsection{This paper}

Our main idea is that statistical physics may provide a correct language and concepts to study the practical behavior of the LLL algorithm. As we demonstrate throughout this paper, for each LLL phenomenon, there is a corresponding sandpile phenomenon, which can be captured and then studied with the already well-established methods of physics. Indeed, from the physical perspective, there is no reason not to regard the LLL algorithm as a proper member of the family of \emph{stochastic sandpile models}. Since the identification of the correct language often does much good in mathematics, this may open up a path to a systematic understanding of the various unexplained phenomena of LLL, via the mathematics and physics of the sandpile models.

Some similarities of LLL to sandpiles have been noticed previously and utilized to some extent. To the best of our knowledge, this is first pointed out in Madritsch and Vall\'ee (\cite{MV10}), and also in Vall\'ee (\cite{V16}), albeit briefly. Some aspects of this analogy have also been applied to BKZ as well --- see \cite{HPS11} and \cite{BSW18} for instance.

One of the new contributions made by the present paper is the introduction of stochastic sandpile models that are both impressively close to LLL (precisely speaking, its Siegel variant) and mathematically accessible. We propose two models of LLL, which we name \emph{LLL-SP} and \emph{SSP} respectively. LLL-SP (Algorithm \ref{alg:lllsp}) is a sandpile model that exhibits nearly identical quantitative behavior to that of LLL in many aspects, suggesting that the two algorithms operate under the same principles. This claim can be formulated in a precise language, in terms of the mixing property of the $\mu$ variables; see Conjecture \ref{conj:mu} below. SSP (Algorithm \ref{alg:ssp}) is a simple stochastic variant of the abelian sandpile model (ASM; Algorithm \ref{alg:asm}) that serves as a useful toy model for LLL that is mathematically far more tractable than LLL-SP, and still imitates the important aspects of the output statistics of LLL.

We also demonstrate that it is possible to prove some of the most desired statements for LLL on these models. On SSP, we can establish an upper bound on the average-case RHF that is significantly smaller than the worst-case (Theorem \ref{prop:ssprhf}; the proof is deferred to \cite{Kprep}). This seems much harder to achieve on LLL-SP, but here we are still able to provide a probabilistic lower bound on the time complexity (Theorem \ref{thm:lowertime}), whose order matches the well-known upper bound, and to resolve the optimal LLL problem (Theorem \ref{thm:optimal}). These results support our idea that the sandpile interpretation may be the correct approach to the real-world behavior of LLL.

In addition, it is worth noting that this physical perspective on LLL provides convincing heuristics on a number of LLL behavior, which are scattered throughout this paper. For instance, it explains why the output statistics of LLL appears independent of the input distribution, and has the shape as described by the GSA and its failure at the boundaries. Given that there has been not even a vague heuristic accounting for most practical behavior of LLL, this provides yet another good reason to pursue this line of development. Moreover, we try to capture the essence of these heuristics with formal language, with Conjecture \ref{conj:mu} below and the ``parallelepiped argument'' in the proof of Theorem \ref{prop:ssprhf} (also see Conjecture \ref{conj:mass}). These provide some concrete paths for further research.

\subsection{Cryptographic considerations}

LLL is of fundamental importance to lattice-based cryptography, and for this reason alone it deserves to be understood well. Our understanding of LLL may affect our understanding of all other reduction algorithms, particularly BKZ, the current standard method for challenging lattice-based systems.

Specifically, there are certain questions regarding LLL --- and any reduction algorithm in general --- that may pose a threat to lattice-based cryptography as a whole. For instance, it could be that a yet undiscovered small trick may improve the RHF of LLL without meaningfully increasing the time complexity, say, to $1.002$. Absurd as this may sound at first, recall that we already ``improved'' LLL from (RHF) $\approx 1.075$ to $\approx 1.02$ merely by implementing it, and that currently we lack the device to form even a vague argument against such catastrophic possibility.

It seems that the practitioners of lattice-based cryptography are well aware of such uncertainties as to the scope and limit of reduction algorithms, and are reacting accordingly. According to Tables 5--10 of \cite{ACD+18}, most lattice-based submissions to the recent NIST call for proposals (\cite{NIST}) claim about half or less as may bits of security as estimated with the state-of-art techniques. On the other hand, one observes no such reservation in submissions from multivariate cryptography, for example. By developing a framework for a systematic study of reduction algorithms, we hope to be able to relieve some of these concerns in the area of lattice-based cryptography.

\subsection{Assumptions and notations}

Throughout this paper, instead of the original LLL reduction from \cite{LLL82}, we work with its Siegel variant, a slight simplification of LLL. The Siegel reduction shares with LLL all its idiosyncrasies, and a bit easier to handle technically; hence a reasonable starting point for our research.

$n$ always means the dimension of the relevant Euclidean space. Our lattices in $\mathbb{R}^n$ always have full rank.

A basis $\mathcal{B}$, besides its usual definition, is an \emph{ordered} set, and we refer to its $i$-th element as $\vb_i$. Denote by $\vb_i^*$ the component of $\vb_i$ orthogonal to all vectors preceding it, i.e. $\vb_1, \ldots, \vb_{i-1}$. Also, for $i > j$, define $\mu_{i,j} = \langle \vb_i, \vb_j^* \rangle / \langle \vb_j^*, \vb_j^* \rangle$. Thus the following equality holds in general:
\begin{equation*}
\vb_i = \vb_i^* + \sum_{j=1}^{i-1} \mu_{i,j}\vb_j^*.
\end{equation*}

We will write for shorthand $\alpha_i := \|\vb_i^*\| / \|\vb_{i+1}^*\|$, and $Q_i = (\alpha_i^{-2} + \mu_{i+1,i}^2)^{-1/2}$. When discussing lattices, $r_i := \log \alpha_i$, and when discussing sandpiles, $r_i$ refers to the ``amount of sand'' at vertex $i$.

\section{Modeling LLL by a sandpile}

\subsection{The LLL algorithm}
We briefly review the LLL algorithm; for details, we recommend \cite{LLL82}, in which it is first introduced, and also \cite{JS98} and \cite{NV10}. A pseudocode for the LLL algorithm is provided in Algorithm \ref{alg:lll}. Be reminded that, whenever we mention LLL, we are really referring to its Siegel variant.
%In Line 3, we deliberately left the \emph{choice algorithm}, that is, the method for choosing $k$, undecided.

\begin{algorithm}
\caption{The LLL algorithm (Siegel variant)}\label{alg:lll}
\begin{enumerate}[1.]
\item[0.] Input: a basis $\mathcal{B} = \{\vb_1, \ldots, \vb_n\}$ of $\mathbb{R}^n$, a parameter $\delta < 0.75$
\item while true, do:
\item \hspace{4mm} Size-reduce $\mathcal{B}$.
\item \hspace{4mm} (Lov\'asz test) choose the lowest $k \in \{1, \ldots, n-1\}$ such that $\delta\|\vb_{k}^*\|^2 > \|\vb_{k+1}^*\|^2$
\item \hspace{4mm} if there is no such $k$, break
\item \hspace{4mm} swap $\vb_{k}$ and $\vb_{k+1}$ in $\mathcal{B}$
\item Output $\mathcal{B} = \{\vb_1, \ldots, \vb_n\}$, a $\delta$-reduced LLL basis.
\end{enumerate}
\end{algorithm}

\begin{proposition} \label{prop:post_swap}
After carrying out Step 5 in Algorithm \ref{alg:lll}, the following changes occur:
\begin{enumerate}[(i)]
\item $\alpha_{k-1}^{new} = Q_k\alpha_{k-1}$
\item $\alpha_k^{new} =  Q_k^{-2}\alpha_k$
\item $\alpha_{k+1}^{new} =  Q_k\alpha_{k+1}$
\item $\mu_{k, k-1}^{new} =  \mu_{k+1, k-1}$
\item $\mu_{k+1, k}^{new} =  Q_k^2\mu_{k+1,k}$
\item $\mu_{k+2, k+1}^{new} =  \mu_{k+2,k} - \mu_{k+2,k+1}\mu_{k+1,k}$
\item $\mu_{k,l}^{new} = \mu_{k+1,l}, \mu_{k+1,l}^{new} = \mu_{k,l}$ for $1 \leq l \leq k-1$
\item $\mu_{l,k}^{new} = \mu_{l,k+1} - \mu_{l,k+1}\mu_{k+1,k}\mu_{k+1,k}^{new} + \mu_{l,k}\mu_{k+1,k}^{new}$ for $l \geq k+2$
\item $\mu_{l,k+1}^{new} = \mu_{l,k} - \mu_{l,k+1}\mu_{k+1,k}$ for $l \geq k+2$
\end{enumerate}
and there are no other changes. The superscript ``new'' refers to the corresponding variable after the swap.
\end{proposition}
\begin{proof}
Straightforward calculations (see e,g, \cite{LLL82}).
\end{proof}

\subsection{Sandpile basics}

We also briefly review the basics of the sandpile models. For references, see Dhar (\cite{D99}, \cite{D06}) or Perkinson (\cite{P14}). 

A sandpile model is defined on a finite graph $\mathcal{G}$, with one distinguished vertex called the \emph{sink}. In the present paper, we only concern ourselves with the cycle graph, say $A_n$, consisting of vertices $\{v_1, \ldots, v_n\}$ and one unoriented edge for each adjacent pair $v_i$ and $v_{i+1}$. We also consider $v_1$ and $v_n$ as adjacent. We designate $v_n$ as the sink.

A \emph{configuration} is a function $r : \{v_1, \ldots, v_n\} \rightarrow \mathbb{R}$. Just as reduction algorithms work with bases, sandpile models work with configurations. We write for short $r_i = r(v_i)$. One may think of $r_i$ as the amount or \emph{height} of the pile of sand placed on $v_i$.

\begin{figure}
\includegraphics[scale=1.2]{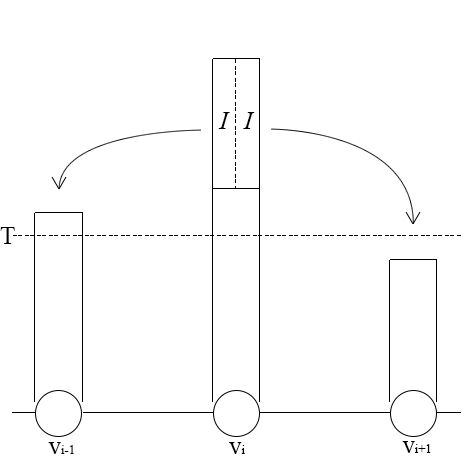}
\caption{An illustration of a (legal) toppling $T_i$.} \label{fig:toppling}
\end{figure}

Just as LLL computes a reduced basis by repeatedly swapping neighboring basis vectors, sandpiles compute a \emph{stable configuration} by repeated \emph{toppling.} Let $T, I \in \mathbb{R}_{>0}$. A configuration is \emph{stable} if $r_i \leq T$ for all $i \neq n$. A \emph{toppling operator} $T_i$ ($i \neq n$) replaces $r_i$ by $r_i - 2I$, and $r_{i-1}$ by $r_{i-1} + I$ and $r_{i+1}$ by $r_{i+1} + I$. An illustration is provided in Figure \ref{fig:toppling}. Applying $T_i$ when $r_i > T$ is called a \emph{legal toppling}. By repeatedly applying legal topplings, all excess ``sand'' will eventually be thrown away to the sink, and the process will terminate.

In our paper, $T$ --- \emph{threshold} --- will always be a fixed constant, but $I$ --- \emph{increment} --- could be a function of the current configuration, or a random variable, or both. In the former case, we say that the model is \emph{nonabelian} --- otherwise \emph{abelian}. In the second case, we say that the model is \emph{stochastic}. The (non-stochastic) abelian sandpile theory is quite well-developed, with rich connections to other fields of mathematics --- see e.g. \cite{L10}. Other sandpile models are far less understood, especially the nonabelian ones.

\subsection{The LLL sandpile model}

Motivated by Proposition \ref{prop:post_swap}, especially the formulas (i) -- (iii), we propose the following Algorithm \ref{alg:lllsp}, which we call the \emph{LLL sandpile model}, or LLL-SP for short. 

\begin{algorithm}
\caption{The LLL sandpile model (LLL-SP)}\label{alg:lllsp}
\begin{enumerate}[1.]
\item[0.] Input: $\alpha_1, \ldots, \alpha_{n} \in \mathbb{R}$, $\mu_{2,1}, \ldots, \mu_{n,n-1} \in [-0.5,0.5]$, a parameter $\delta < 0.75$
\item Rewrite $r_i : = \log \alpha_i$, $\mu_i := \mu_{i+1,i}$ $T := -0.5\log \delta$
\item while true, do:
\item \hspace{4mm} choose the lowest $k \in \{1, \ldots, n-1\}$ such that $r_k > T$
\item \hspace{4mm} if there is no such k, break
\item \hspace{4mm} subtract $2\log Q_k$ from $r_k$
\item \hspace{4mm} add $\log Q_k$ to $r_{k-1}$ (if $k-1 \geq 1$) and $r_{k+1}$ (if $k+1 \leq n-1$)
\item \hspace{4mm} (re-)sample $\mu_{k-1}, \mu_k, \mu_{k+1}$ uniformly from $[-0.5,0.5]$
\item Output: real numbers $r_1, \ldots, r_{n-1} \leq T$
\end{enumerate}
\end{algorithm}

The only difference between LLL (Algorithm \ref{alg:lll}) and LLL-SP (Algorithm \ref{alg:lllsp}) lies in the way in which the $\mu$'s are replaced after each swap or topple. Our experimental results below demonstrate that this change hardly causes any difference in their behavior. A theoretical perspective is discussed at the end of this section.

The decision to sample $\mu_i$'s uniformly is largely provisional, though some post hoc justification is provided in Figure \ref{fig:seq}. One could refine the model by updating $\mu_i$'s with the formulas in Proposition \ref{prop:post_swap}, and then re-sampling $\mu_{i+2,i}$'s uniformly.

\subsection{Numerical comparisons}

Figure \ref{fig:output} shows the average shape of the output bases and configurations by LLL and LLL-SP. We also ran the same experiments where the method of choosing $k$ is tweaked, e.g. choose randomly among eligible $k$'s. We omitted them here due to limited space, but the same lessons are obtained anyway.

For each dimension $n = 80, 100, 120$, we ran each algorithm 5,000 times with the same set of input bases of determinant $\approx 2^{10n}$, generated using the standard method suggested in Section 3 of \cite{NS06}. A point $(i, y)$ in each plot indicates that the average of $r_i := \log \alpha_i$'s over those 5,000 outputs equal $y$. We used fpLLL (\cite{FPLLL}) for the LLL algorithm.

One easily observes that the algorithms yield nearly indistinguishable outputs. In particular, since RHF can be computed directly from the $r_i$'s by the formula

\begin{equation} \label{eq:rhf}
\mbox{RHF} = \exp\left(\frac{1}{n^2}\sum_{i=1}^{n-1}(n-i)r_i\right),
\end{equation}
we expect both algorithms to yield about the same RHF. Indeed, Figure \ref{fig:RHFdist} shows that the RHF distribution of LLL and LLL-SP are in excellent agreement.

\begin{figure} 
\centering
\includegraphics[scale=0.30]{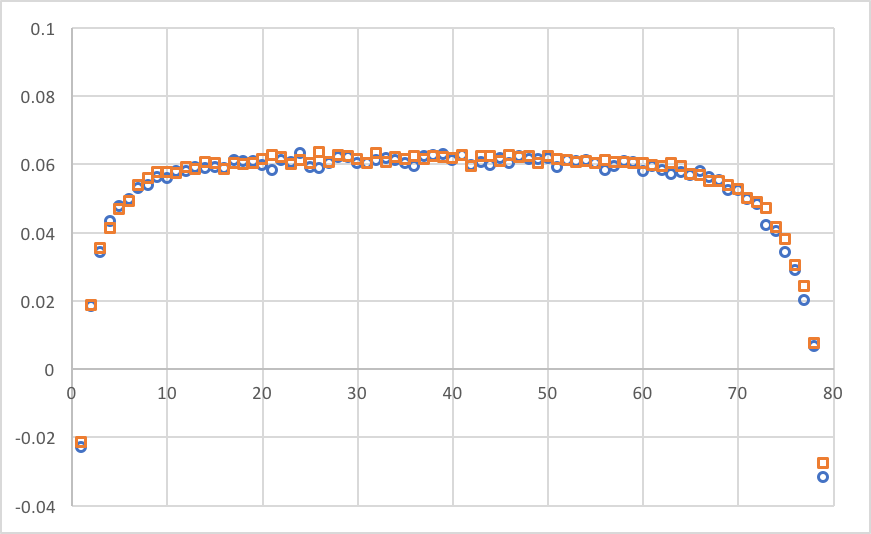} %\includegraphics[scale=0.30]{80r} \includegraphics[scale=0.30]{80g}
\includegraphics[scale=0.30]{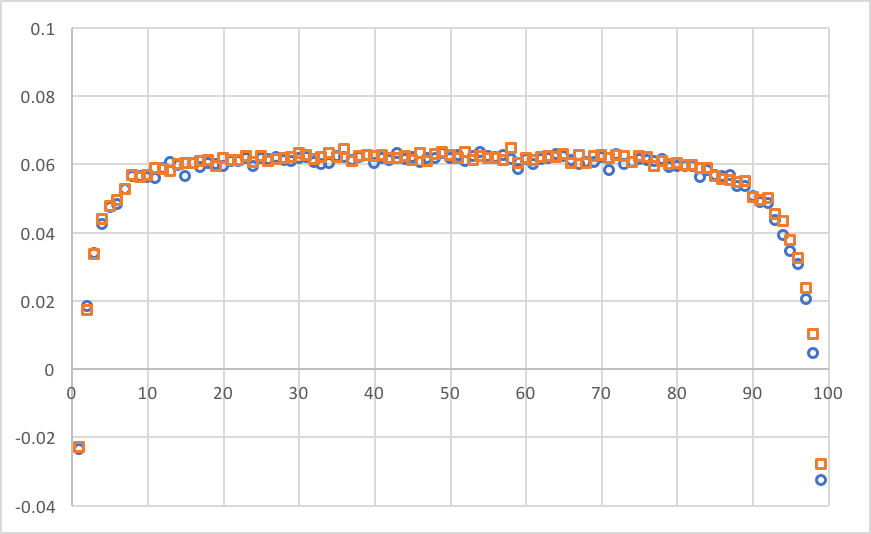} %\includegraphics[scale=0.30]{100r} \includegraphics[scale=0.30]{100g}
\includegraphics[scale=0.30]{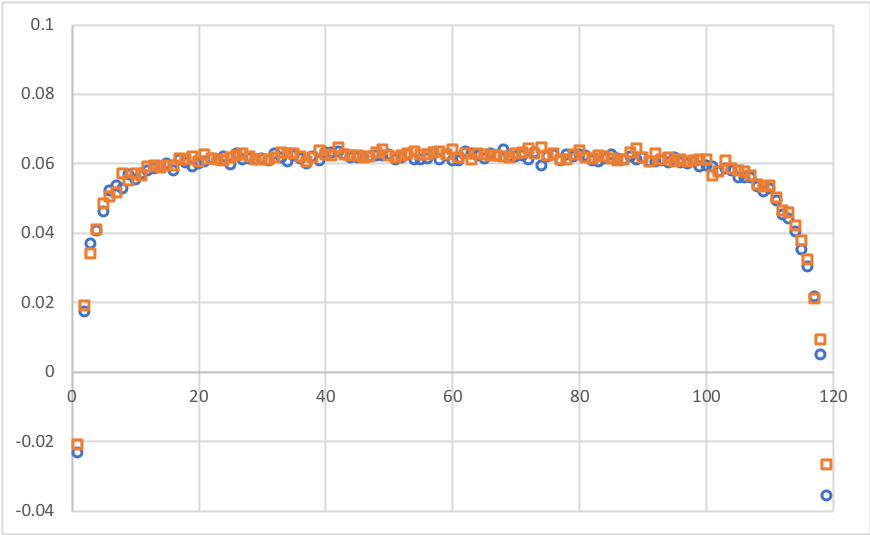} %\includegraphics[scale=0.30]{120r} \includegraphics[scale=0.30]{120g}
\caption{Average output of LLL (orange square) and LLL-SP (blue circle) in dimensions 80, 100, and 120.
%Graphs on each column, from left to right, correspond to the sequential, random, and greedy choice algorithms, respectively. 
%Graphs on each row represent the results in dimensions 80, 100, and 120, respectively.
} \label{fig:output}
\end{figure}

\begin{figure} 
\centering
\includegraphics[scale=0.40]{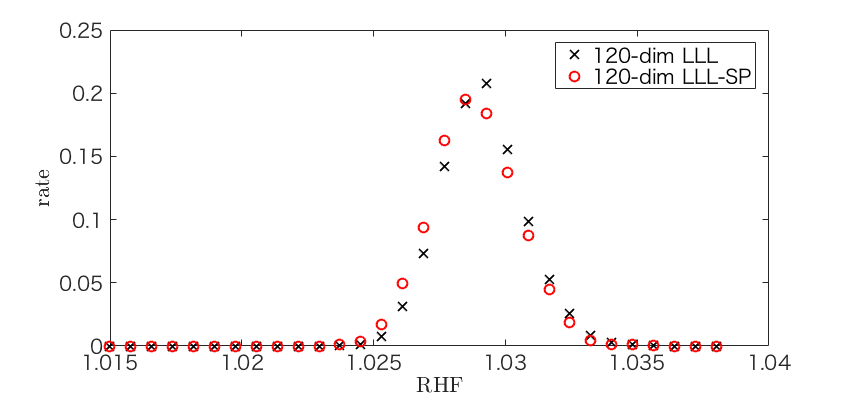}
\caption{RHF distributions of LLL and LLL-SP in dimension 120. We obtain the same results in other dimensions.} \label{fig:RHFdist}
\end{figure} 

The resemblance of the two algorithms runs deeper than on the level of output statistics. See Figure \ref{fig:seq}, which depicts the plot of points $(i, Q_{k(i)}^{-2})$ and $\mu_{k(i)+1, k(i)} = \mu_{k(i)}$ as we ran LLL and LLL-SP on dimension 80, where $k(i)$ is $k$ chosen at $i$-th iteration.\footnote{We have the same results in higher dimensions, but they are too cumbersome to present here.} The two plots are again indistinguishable, yet another evidence that LLL and LLL-SP possess nearly identical dynamics.

%There is a small caveat: the input has to be generic in a certain sense, or otherwise LLL and LLL-SP will behave differently. For example, if the underlying lattice of an input basis has a large gap in the $i$-th and $i+1$-st successive minima, this does not affect LLL-SP in any way, yet LLL would treat the input as essentially two bases of dimension $i$ and $n-i$ respectively, yielding a better output. 

\begin{figure} 
\centering
\includegraphics[scale=0.47]{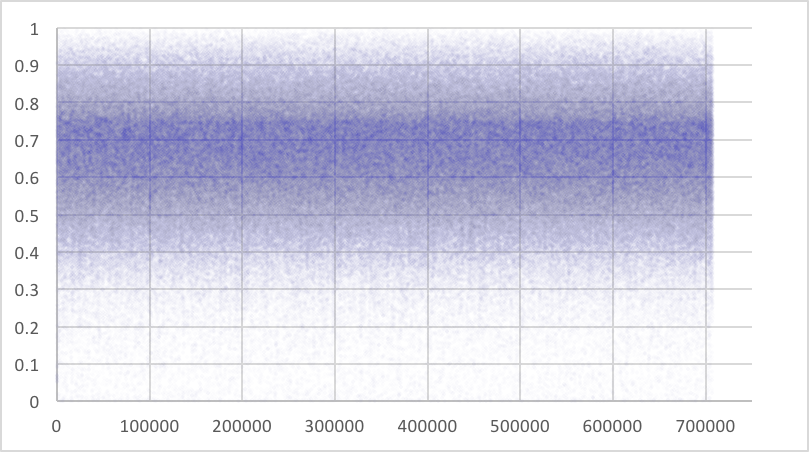} \includegraphics[scale=0.47]{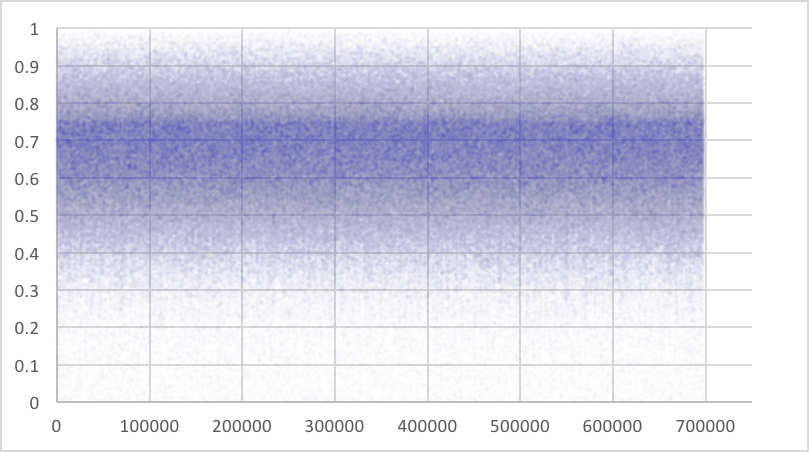}

\includegraphics[scale=0.47]{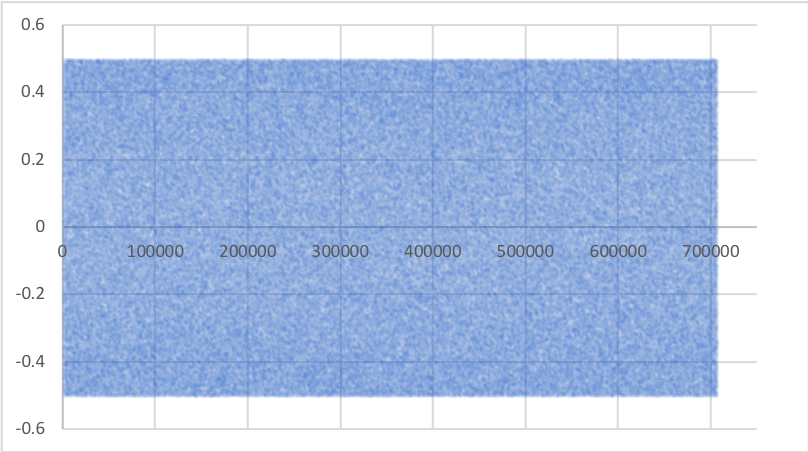} \includegraphics[scale=0.47]{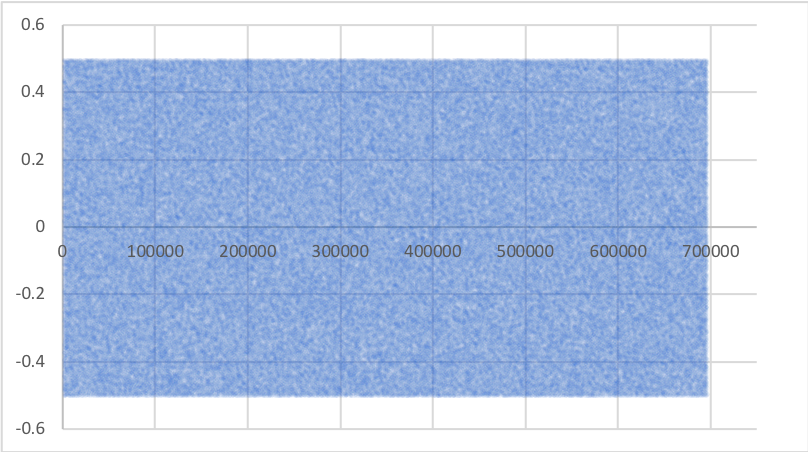}

\caption{Plot of $Q_{k(i)}^{-2}$(top) and $\mu_{k(i)+1, k(i)} = \mu_{k(i)}$(bottom) during a typical run of LLL(left) and LLL-SP(right) in dimension 80.} \label{fig:seq}
\end{figure} 

\subsection{Discussion}

The only difference between LLL and LLL-SP has to do with the way they update the $\mu_k (= \mu_{k+1,k})$'s. For LLL-SP, the $\mu_k$-variables are i.i.d. and independent of the $r_k$-variables. For LLL, $\mu_k$ is determined by a formula involving its previous value and $r_k$. However, it seems plausible that the $\mu_k$'s in LLL, as a stochastic process, is \emph{mixing}, which roughly means that they are close to being i.i.d, in the sense that a small perturbation in $\mu_k$ causes the next value $\mu_k^{new}$ to become near unpredictable. Numerically, this is robustly supported by the graphs at the bottom of Figure \ref{fig:seq}. Theoretically, our intuition comes from the fact that the formula $\mu_{k}^{new} = \mu_{k}/(\mu_k^2 + \alpha_k^{-2})$ (mod 1) is an approximation of the Gauss map $x \mapsto \{1/x\}$, which is well-known to have excellent mixing properties (see e.g. Rokhlin (\cite{R61}) and the references in Bradley (\cite{B05}) for more recent works).

The above discussion can be summarized and formulated in the form of a mathematical conjecture, which can then be considered a rigorous version of the statement ``LLL is essentially a sandpile model.'' Below is our provisional formulation of such a conjecture.

\begin{conjecture} \label{conj:mu}
Choose a distribution $\mathcal{D}$ on the set of bases in $\mathbb{R}^n$, to be used to sample inputs for LLL. Define $k(i)$, as earlier, to be the index of the pile toppled at $i$-th iteration. Then $k(i)$ is a random variable depending on the input distribution, and so is $\mu_{k(i)}$. Then, if $\mathcal{D}$ is ``generic,'' then

\begin{enumerate}[(i)]
\item $(|\mu_{k(i)}|)_{i = 1, 2, \ldots}$ is strongly mixing as a stochastic process. (Roughly speaking, this means $|\mu_k(N)|$ is nearly independent of $|\mu_k(M)|$ for which $N-M$ is large; see the text \cite{B95} for a precise definition.)
\item each $|\mu_{k(i)}|$ is contained in a compact subset $S$ of the set of all probability density functions on $[0, 0.5]$ with respect to the $L^\infty$-norm. $S$ is independent of the dimension, the input distribution, or any other variable.
\end{enumerate}
\end{conjecture}

Ideally, Conjecture \ref{conj:mu} is to be designed so that what is provable for LLL-SP would also be provable for LLL by an analogous argument (e.g. the theorems in Section 4), while retaining the flexibility as to what the correct distribution of $\mu_k$ might be. It is to be updated accordingly as our understanding of LLL and LLL-SP progresses. Then at some point Conjecture \ref{conj:mu} may come within reach of mathematics; or maybe it already is.

As for the expression ``generic'' in the statement of Conjecture \ref{conj:mu}, it is hard to pin down its precise meaning at this point, as is sometimes the case in mathematics. But there are two criteria that a generic $\mathcal{D}$ must fulfill, which we suspect are also sufficient conditions. First, most of the samples from $\mathcal{D}$ must not ruin the mixing property of $\mu^{new}_k$; a counterexample is a basis of a lattice with a huge discrepancy between two successive minima $\lambda_k \ll \lambda_{k+1}$, which would cause $\mu_{k(i)}$ to be abnormally small whenever $k(i) = k$. The second criterion concerns the shape of $\mathrm{supp}\,\mathcal{D}$, which we illustrate in the next section.

\section{Abelian sandpile analogue of LLL}

The drawback of LLL-SP as a model of LLL is that, being nonabelian, it is difficult to study. The existing literature on nonabelian sandpile models is rather thin, and to the best of our knowledge, there is no theoretical treatment on this subject. We hope that our work here provides some motivation to pursue nonabelian models in depth.

In this section, we introduce a certain abelian stochastic sandpile model that we named SSP, which is in a sense an abelianized version of LLL-SP. A priori, SSP appears completely unrelated to LLL. Surprisingly, though, its average output shape turns out to be extremely close to that of LLL. Moreover, SSP admits a mathematical theory that is analogous to that of ASM developed by Dhar (\cite{D90}, see also \cite{D06}), which is developed in a separate paper by the second-named author (\cite{Kprep}). This allows one to prove some pleasant statements such as Theorem \ref{prop:ssprhf}. Therefore, SSP is a useful toy model that could yield insights into some of the most prominent features of the output statistics of LLL. It may yield some hints as to how to start analyzing the non-abelian LLL-SP as well.

\subsection{Background on ASM}

To facilitate reader's understanding, we briefly describe the abelian sandpile model (ASM), the most basic of sandpile models, and parts of its theory that is relevant to us. Its pseudocode is provided in Algorithm \ref{alg:asm}. See Dhar (\cite{D90}), where the theory is originally developed, or Perkinson (\cite{P14}) for an articulate and readable exposition.

\begin{algorithm}
\caption{Abelian sandpile model (ASM)}\label{alg:asm}
\begin{enumerate}[1.]
\item[0.] Input: $r_1, \ldots, r_{n-1} \in \mathbb{Z}$, parameters $T, I \in \mathbb{Z}$, $0 < I \leq T/2$
\item while true, do:
\item \hspace{4mm} choose a $k \in \{1, \ldots, n-1\}$ such that $r_k > T$
\item \hspace{4mm} if there is no such k, break
\item \hspace{4mm} subtract $2I$ from $r_k$
\item \hspace{4mm} augment $I$ to $r_{k-1}$ and $r_{k+1}$
\item Output: integers $r_1, \ldots, r_{n-1} \leq T$
\end{enumerate}
\end{algorithm}

The important ASM concepts for us are that of the \emph{recurrent configurations} and the \emph{steady state}. Let $M$ be the set of all stable (non-negative) configurations of ASM. Given two configurations $r, s \in M$, we have the operation
\begin{equation*}
r \oplus s = \mbox{(stabilization of $r+s$)},
\end{equation*}
which is the outcome of ASM with input being the configuration $r+s$ defined by $(r+s)_i = r_i + s_i$ for each $i$. Unlike LLL, the output of ASM is independent of the choice of toppling order --- hence the term ``abelian'' --- and thus $\oplus$ is well-defined. This operation makes $M$ into a commutative monoid.

Define $g \in M$ to be the configuration with $g_1 = 1$ and $g_2 = \ldots = g_{n-1} = 0$. We call $r \in M$ \emph{recurrent} if
\begin{equation*}
\underbrace{g \oplus \ldots \oplus g}_{m\, \mathrm{times}} = \mbox{$r$ for infinitely many $m$}.
\end{equation*}

One can actually take any $g$ for which at least one $g_i$ is coprime to the g.c.d. of $T$ and $I$ (this condition is nothing but only to avoid concentration on a select few congruence classes). Equivalently, with LLL in mind, we can also define that $r$ is recurrent if there exist infinitely many input configurations such that their stabilization results in $r$. It is a theorem that the set $R$ of the recurrent configurations of ASM forms a group under $\oplus$.

One may ask, given an $r \in R$, what is the proportion of $m \in \Z_{>0}$ that satisfies $g \oplus \ldots \oplus g\, (m\, \mathrm{times}) = r$? It turns out that the answer is $1/|R|$ for any $r \in R$, that is, each element of $R$ has the same chance of appearing. This distribution, say $\rho$, on $R$ is called the \emph{steady state} of the system. And the phrase \emph{average output shape} that we have been using in the empirical sense obtains a formal definition as $\sum_{r \in R} \rho(r)r$. The steady state is unique in the following sense: choose an $r \in R$ according to $\rho$, and take any configuration $s$; then the stabilization of $r+s$, which is a random variable, has distribution $\rho$.

Equivalently, this means that, in terms of our second definition of recurrent, if we sample the input configuration for ASM from a ``generic'' distribution e.g. uniformly from a large rectangular set of $\Z_{>0}^{n-1} \cong$ (the space of nonnegative configurations), we would obtain each $r \in R$ with about probability $\rho(r)$. To elaborate, in the case of ASM, the set of inputs whose output becomes a given $r \in R$ is contained in a coset $r + \Lambda$ of a sublattice $\Lambda \subseteq \Z^{n-1}$. So the word ``generic'' here means ``containing about the same number of each coset representatives of $\Lambda$.'' This is essentially the second condition for genericity that we mentioned in the discussion after Conjecture \ref{conj:mu}.

\subsection{Introduction to SSP}

A pseudocode for SSP is provided in Algorithm \ref{alg:ssp}. This is exactly the same as ASM, except for Step 4, which determines the amount of sand to be toppled at random. The decision to sample from the uniform distribution is an arbitrary one; we could have chosen something else, and much of the discussion below still apply.

\begin{algorithm}
\caption{Stochastic sandpile (SSP)}\label{alg:ssp}
\begin{enumerate}[1.]
\item[0.] Input: $r_1, \ldots, r_{n-1} \in \mathbb{Z}$, parameters $T, I \in \mathbb{Z}$, $0 < I \leq T/2$
\item while true, do:
\item \hspace{4mm} choose a $k \in \{1, \ldots, n-1\}$ such that $r_k > T$
\item \hspace{4mm} if there is no such k, break
\item \hspace{4mm} sample $\gamma$ uniformly from $\{1, \ldots, I\}$
\item \hspace{4mm} subtract $2\gamma$ from $r_k$
\item \hspace{4mm} augment $\gamma$ to $r_{k-1}$ and $r_{k+1}$
\item Output: integers $r_1, \ldots, r_{n-1} \leq T$
\end{enumerate}
\end{algorithm}

The average output shape of this stochastic sandpile model (SSP) is shown in Figure \ref{fig:sspoutput}. Figure \ref{fig:sspoutput} not only shares all the major characteristics of Figure \ref{fig:output}, but they are also quantitatively alike. The values $r_i$'s are nearly identical in the middle, which gradually decreases as $i$ approaches the boundary, starting at around $i = 15$ and $n - 15$. Furthermore, in both figures, the differences between the threshold and the middle values, and the differences between the middle and the boundary value, are equal; in the case of SSP, it is $\approx I/4$, and for LLL, it equals about $0.08$. Outside the paradigm we are developing here, this should come as quite surprising. Algorithms \ref{alg:lll} and \ref{alg:ssp} are simply so different that there is no reason to expect that anything similar would come out of them.

This finding suggests that sandpiles may shed light on the geometric series assumption(GSA) and its partial failure at the boundary, yet another important unexplained phenomenon of LLL-like reduction algorithms. Indeed, ``GSA'' is a quite general phenomenon that is fairly well-understood in statistical physics, via the finite-size scaling theory (see e.g. \cite{GDM16}). We will explore this connection in a forthcoming paper.

\begin{figure}
\centering
\includegraphics[scale=0.6]{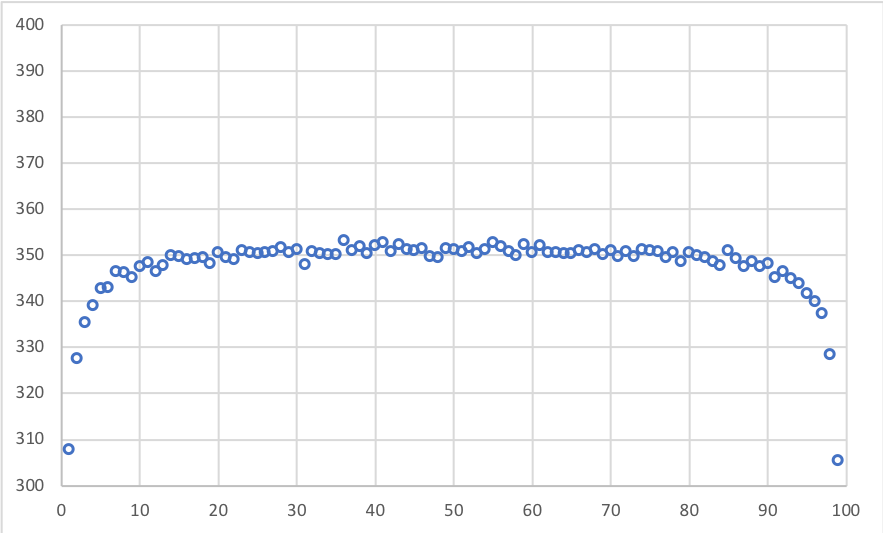}
\caption{Average output of SSP, $n = 100$, $I = 200$ and $T = 400$.} \label{fig:sspoutput}
\end{figure}

\subsection{Mathematical properties of SSP}

A mathematical theory of SSP closely analogous to that of ASM has been recently developed (\cite{Kprep}), largely motivated by the experimental result above. Every aspect of the above-mentioned ASM theory has its appropriate counterpart in the SSP theory, except that instead of configurations one works with a distribution on the set of configurations, due to its stochastic nature. Most importantly, SSP, like ASM, possesses the unique steady state. Figure \ref{fig:sspoutput} is a reflection of that steady state of SSP, from which one can compute its average ``RHF'' via \eqref{eq:rhf}.

As some readers might have noticed in the earlier discussion, the notion of the steady state alone clarifies numerous aspects of the practical behavior of LLL. The number ``$1.02$'' obtains a neat mathematical meaning as a certain invariant of the steady state. Moreover, that the steady state is unique --- i.e. it is the only attractor in the associated dynamics --- explains why the number $1.02$ seems independent of how the input bases are sampled for LLL, an observation in \cite{NS05} that is also declared as a conjecture there. Therefore, most of the frequently asked questions regarding the average output of LLL --- why $1.02$, why GSA, why independent of input sampling --- comes down to studying the quantitative properties of the steady state.

There are two difficulties. First, we do not yet know how to prove that LLL indeed possesses a steady state, even upon fixing the order of toppling. Second, studying the steady state is nontrivial, even for ASM or SSM, where it reduces to tricky problems in combinatorics. Still, we can prove some of the statements for SSP that one wishes of LLL. For example, it is possible to rigorously bound the average RHF of SSP from above:

\begin{theorem}\label{prop:ssprhf}
The worst-case $\log\mathrm{(RHF)}$ of SSP is $T/2 + o_n(1)$. The average $\log\mathrm{(RHF)}$ of SSP is bounded from above by $T/2 - I/2e^2 + o_n(1)$.
\end{theorem}

We note that empirically one observes $\log\mathrm{(RHF)} \approx T/2 - I/8$ on average.

\begin{proof}[Sketch (and discussion) of proof]
%This is an immediate corollary of Proposition 8 of \cite{Kprep}. The idea is to estimate the number of stable configurations with $\log\mathrm{(RHF)} > T/2 - \alpha$, and multiply it by the maximum point mass $m \approx I^{-(n-1)}$ of the steady state to bound the proportion of such configurations in the steady state. For any $\alpha < 1/me^2$ one can easily show that it approaches zero as $n \rightarrow \infty$.

This is essentially Proposition 8 of \cite{Kprep}. We present the idea of the proof for completeness; it takes only a moderate amount of labor to turn it into a formal argument. 

\begin{figure}
\centering
\includegraphics[scale=0.3]{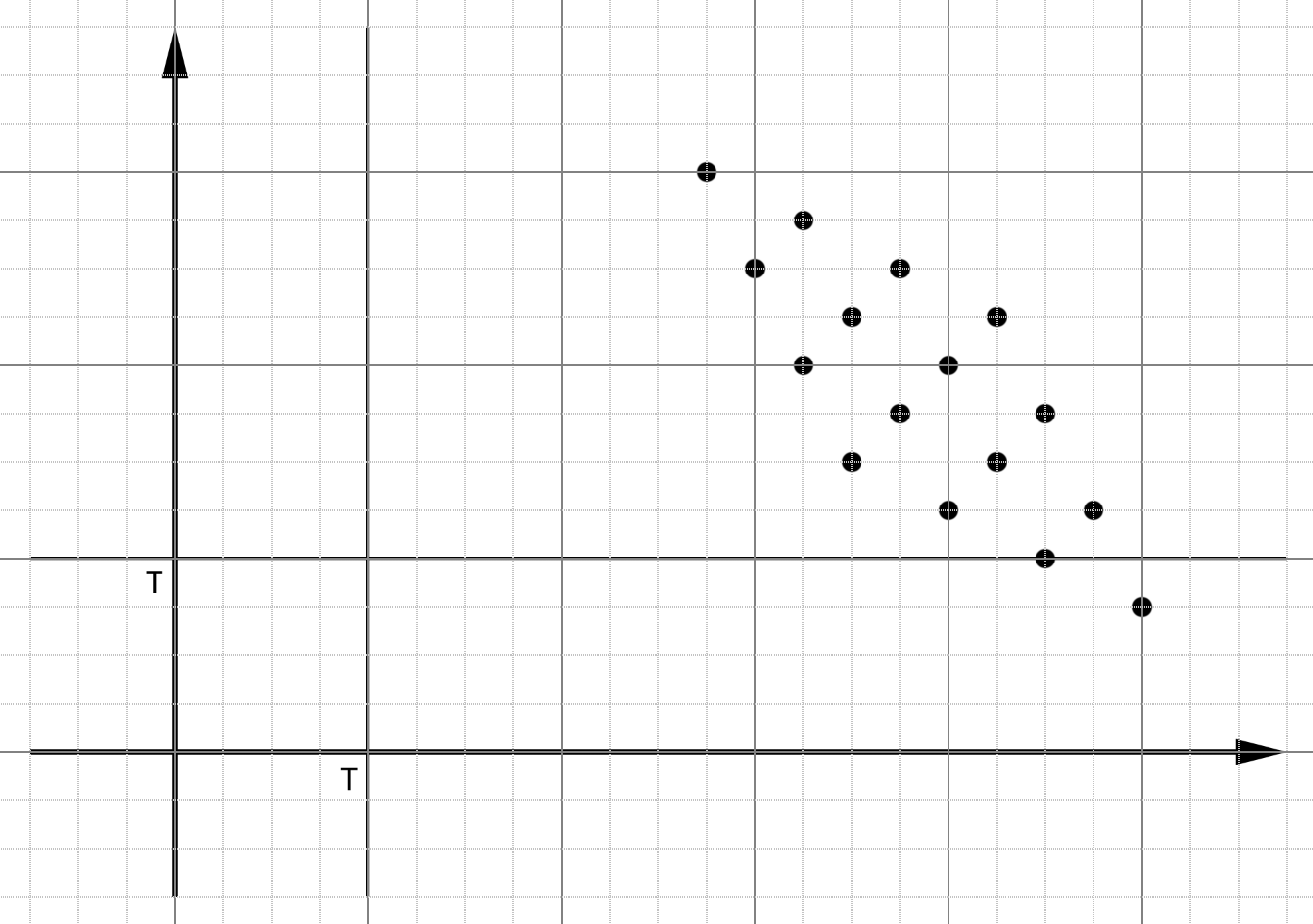} \vspace{3mm}

\includegraphics[scale=0.3]{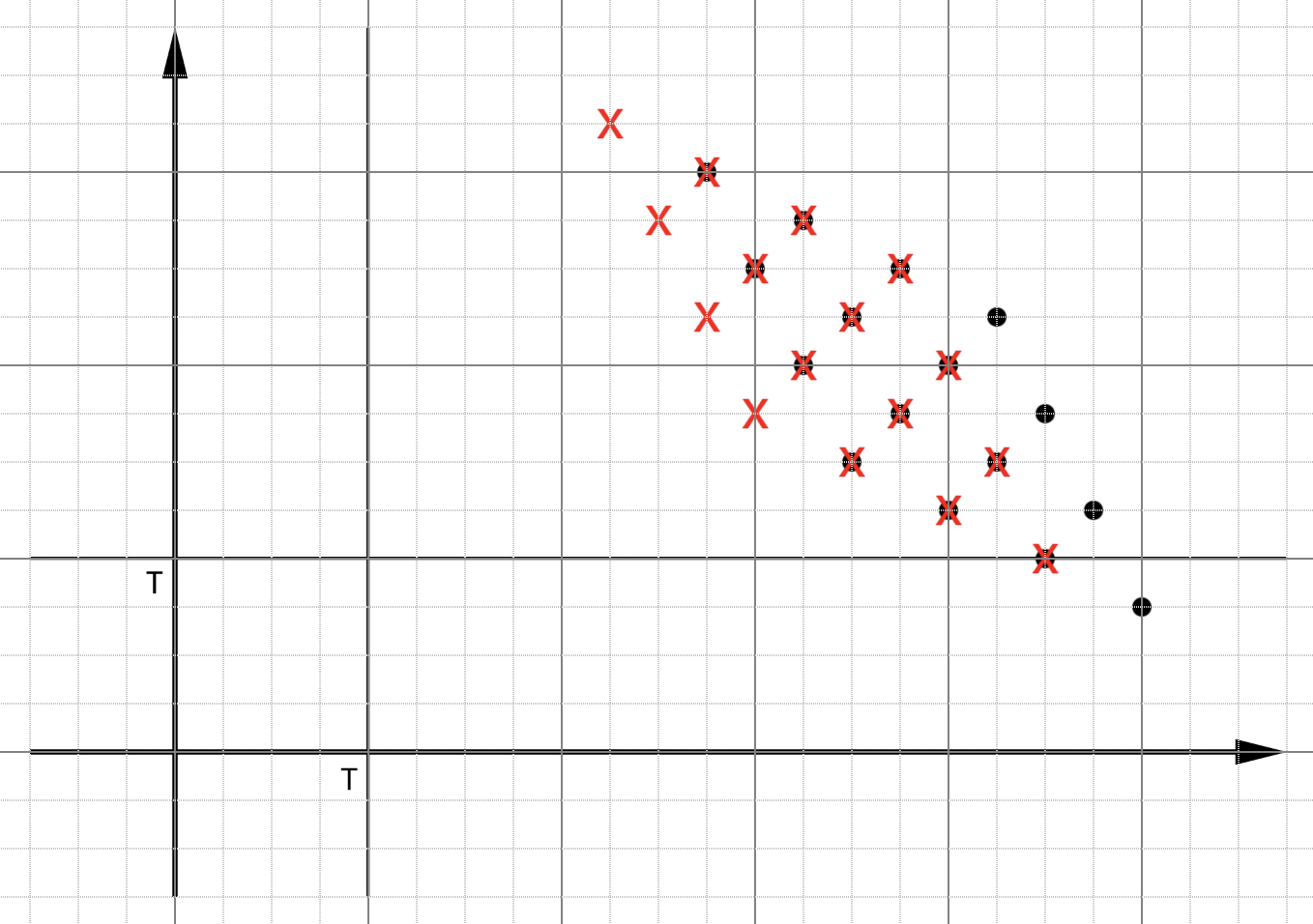} \vspace{3mm}

\includegraphics[scale=0.3]{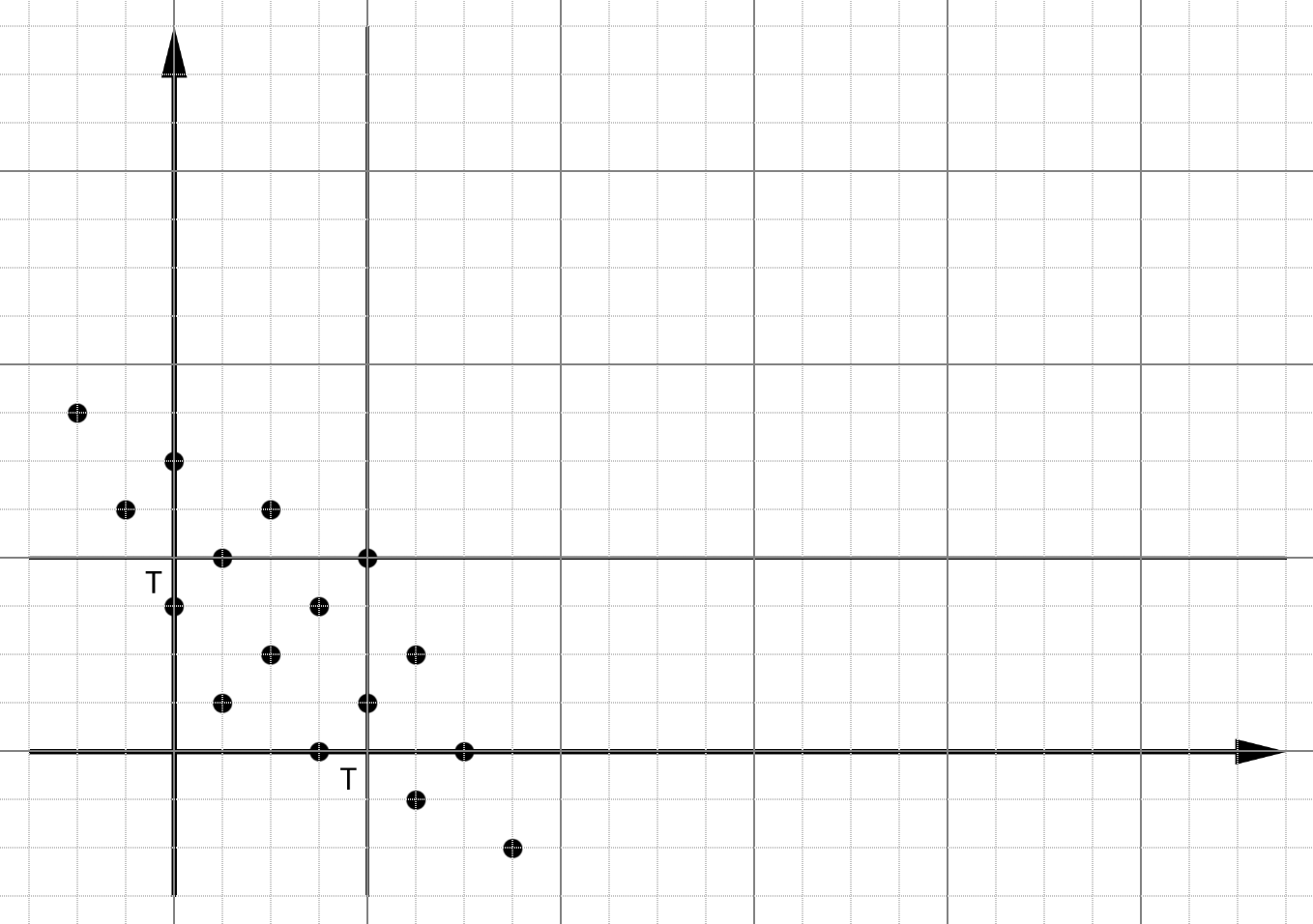} \vspace{2mm}

\caption{The parallelepiped argument.} \label{fig:par}
\end{figure} 

Take an unstable configuration $r$. If $r$ is sufficiently far away from the origin in the configuration space, we must topple on each and every vertex at least once --- in fact, arbitrarily many times --- on the way of stabilizing $r$. So consider $T_1T_2 \ldots T_{n-1}r$, where $T_i$ is the toppling operator on vertex $i$. Since the outcome of SSP is a random variable, $T_1T_2 \ldots T_{n-1}r$ should be thought of as a probability distribution on the configuration space. As such, it is a distribution that is supported on a parallelepiped-shaped cluster, as illustrated in the top of Figure \ref{fig:par} in case $n = 3$ and $I = 4$; the upper-right vertex in the parallelogram is $r - (1, 1, \ldots, 1)$, which is $r$ toppled on every vertex by the minimum possible amount.

Applying $T_i$ to this parallelepiped-shaped distribution amounts to ``pushing'' the parallelepiped in the direction of $i$, resulting in another parallelepiped-shaped distribution. The middle graph in Figure \ref{fig:par} illustrates this process, by indicating with x marks the outcome of applying $T_1$ to the original distribution (assuming that the horizontal axis represents $r_1$). Repeating, we eventually reach the situation as in the bottom of Figure \ref{fig:par}, where none of the $T_i$ would preserve the shape of the parallelepiped, since $(T, T, \ldots, T)$ is already a stable configuration and thus $T_i$ leaves it there. From this point on, it is rather tricky to describe the action of the $T_i$'s, which is the source of the difficulty of studying the steady state of SSP.

However, we claim that, for any $r$ sufficiently far enough from the origin, the distribution on the parallelepiped obtained by the time the upper-right corner reaches $(T, \ldots, T)$ is arbitrarily close to a certain limiting distribution $\wp$. To see this, consider the action of $T_i$ on the distribution on the parallelepiped, while forgetting the information about where that parallelepiped is located in the configuration space. Then one notices that each $T_i$ acts as a linear operator on the space of such distributions. Simultaneously diagonalizing all $T_i$'s --- possible because they pairwise commute --- one finds that $1$ is the single largest eigenvalue of multiplicity one, whose corresponding eigenvector is $\wp$. Upon repeated applications of $T_i$'s, the components corresponding to the lesser eigenvalues converge to zero, proving the claim.

Observe that, when the upper-right corner equals $(T, \ldots, T)$ (again see bottom of Figure \ref{fig:par}), fully stabilizing $\wp$ yields the steady state. $\wp$ itself is easily computed by hand, but the steady state is not. Still, we can prove, using the information about $\wp$, that the maximum density of the steady state occurs at $(T, \ldots, T)$, and that its value equals $\approx (I/2)^{-(n-1)}$. Note that this is enough to deduce a nontrivial upper bound on the average RHF: estimate the number $N(\alpha)$ of stable configurations whose $\log\mbox{(RHF)}$ are greater than $\alpha$, and take $\alpha$ such that $N(\alpha) \cdot (I/2)^{-(n-1)}$ vanishes as $n \rightarrow \infty$. It turns out we can choose $\alpha = T/2 - I/2e^2$.

\end{proof}

There are a couple of difficulties in directly applying the above argument to LLL (even assuming Conjecture \ref{conj:mu}) or LLL-SP. For instance, because the increment depends on the $r_i$'s for those systems, the effect of $T_i$ is not as neat as illustrated in Figure \ref{fig:par}. It would push the side of the parallelepiped with ``uneven force,'' skewing the shape of the parallelepiped. Specifically, for LLL or LLL-SP, the higher $r_i$ is, the greater is the toppling induced by $T_i$, so the upper side is pushed further than the lower side. Technically speaking, since the $T_i$'s depend on the $r_i$-coordinate it is impossible to separate the parallelepiped from the coordinate space as we have done above, and since the $T_i$'s do not commute we cannot use linear algebra to find the attractor of the dynamics. These are the obstacles to proving the existence of the steady state of LLL.

To prove that the ``average'' RHF is bounded strictly away from the worst-case, it is sufficient to show a much weaker statement that the maximum density of the output distribution is not too large. This seems feasible yet quite vexing; we state it as a conjecture below for future reference. As in the SSP case, we expect that the maximum density is attained on the upper-right corner, and the $T_i$'s perturb it at most marginally, once they are iterated sufficiently many times.

\begin{conjecture} \label{conj:mass}
For a generic distribution $\mathcal{D}$ on the set of bases of $\R^n$, the probability density function of the corresponding output distribution $\mathcal{D}^\circ$ of LLL (or LLL-SP) is bounded from above by a constant $C$ that depends only on $n$.
\end{conjecture}

\section{Regarding time complexity}

Although expanding the SSP theory, and Theorem \ref{prop:ssprhf} in particular, to LLL-SP seems challenging for the time being, we are able to prove some attractive statements for LLL-SP with respect to its complexity, which we present below. We also consider their extensions to LLL assuming the truth of Conjecture \ref{conj:mu}.

\subsection{A lower bound}

The theorem below gives a probabilistic \emph{lower} bound on the complexity of LLL-SP, which agrees up to constant factor with the well-known upper bound. There are two ingredients in the proof: (i) measuring the progress of the LLL algorithm by the quantity \emph{energy}, a well-known idea from the original LLL paper (\cite{LLL82}) (ii) bounding the performance of LLL-SP by a related SSP.

\begin{theorem} \label{thm:lowertime}
%Assume Conjecture \ref{conj:mu}. 
Consider LLL-SP, and an input configuration $r$ whose \emph{log-energy} $E = E(r)$, defined by
\begin{equation*}
E(r) = \sum_{j=1}^{n-1} \sum_{i=j}^{n-1} (n-i)r_i,
\end{equation*}
is sufficiently large --- in fact, $E > 10H$ works, with $H$ defined as in \eqref{eq:enbound}. Then the probability that LLL-SP is not terminated in $E/4$ steps is at least $1 - CE^{-1/2}$ for an absolute constant $C > 0$.
\end{theorem}

Observe that the familiar upper bound $O(n^2\log \max_i\|\b_i\|)$ on the number of required steps is equivalent to $O(E)$, with the implicit constant depending on $\delta$.

\begin{proof}
If the algorithm is terminated, then $E$ must have become less than
\begin{equation*}
\sum_{i=1}^{n} (n-i+1)(n-i)T/2,
\end{equation*}
(where $T := -\log \delta^{1/2} > 0$) which equals,
\begin{equation} \label{eq:enbound}
H:= \frac{T}{6}(n^3 - n).
\end{equation}
Taking converse, we see that if $E$ is greater than \eqref{eq:enbound}, then LLL-SP has not yet terminated. At $k$-th toppling, $E$ decreases by at most $\log \mu_{k(i)}^{-2}$, where $k(i)$ is the index of the vertex in which $i$-th toppling occured. If toppled $N$ times, the decrease in $E$ is bounded by at most $F_N := \sum_{i=1}^N \log \mu_{k(i)}^{-2}$. In sum,
\begin{equation} \label{eq:lowertimegoal}
\mathrm{Prob}(E - F_N > H)
\end{equation}
gives the lower bound on the probability that LLL-SP is not terminated after $N$ swaps. Hence, it suffices to show that \eqref{eq:lowertimegoal} is bounded from below by $1-CE^{-1/2}$ when $N = E/2$.

The central limit theorem is applicable on $F_N$, since $\mu_{i(k)}$ are i.i.d.
% explain how this would complete the proof
More precisely, we apply the Berry-Esseen theorem, which asserts the following. Suppose we have i.i.d. random variables $X_1, X_2, \ldots$, so that $m = \mathbb{E}(X_1)$, $\sigma = (\mathbb{E}(X_1^2) - \mathbb{E}(X_1)^2)^{1/2}$, and $\rho = \mathbb{E}(X_1^3)$ are all finite. Furthermore, let $Y_N = \sum_{i=1}^N X_i$, and let $G_N(x)$ be the cumulative distribution function of $Y_N$, and $\Phi_N(x)$ be the cumulative distribution function of the normal distribution $N(Nm, N\sigma^2)$. Then for all $x$ and $N$,
% state it
\begin{equation*}
\left| G_N(x) - \Phi_N(x) \right| = O(N^{-1/2}),
\end{equation*}
where the implied constant depends on $m, \sigma, \rho$ only.

We let $X_i = \log \mu_{k(i)}^{-2}$ so that $F_N = G_N$, and apply the Berry-Esseen. It is easy to compute and check that $m, \sigma, \rho$ are all finite e.g. $m = 2(1+\log 2) \approx 3.386$ and $\sigma = 2$. Then, for a random variable $\mathcal{N}_N \sim N(Nm, N\sigma^2)$, \eqref{eq:lowertimegoal} is bounded by
\begin{equation*}
\mathrm{Prob}(E - \mathcal{N}_N > H)
\end{equation*}
plus an error of $O(N^{-1/2})$.

Now choose $N = E/4$, so that $\mathcal{N}_N \sim N((1+\log 2)E/2, E)$. Using Chebyshev's inequality we can prove
\begin{equation*}
\mathrm{Prob}(\mathcal{N}_N \geq 0.9E) \leq O(E^{-1}),
\end{equation*}
where the implied constant is absolute. Thus if $E$ is large enough so that $E - H > 0.9E$, we have that \eqref{eq:lowertimegoal} is at least $1 - CE^{-1/2}$ for some $C > 0$, as desired. With marginally more effort, it is possible to determine an explicit value for $C$.

\end{proof}

\begin{remark}
1. We can use the same idea to obtain a lower bound on the average RHF of LLL-SP, but it turns out to be slightly less than $1$, which happens to be useless in the context we are in.

2. There exists a central limit theorem for a strong mixing process (\cite{B95}), and also a central limit theorem for a sequence of independent but non-identical sequence of random variables (e.g. the Lyapunov CLT). Conjecture \ref{conj:mu} states that the $|\mu_{k(i)}|$ of LLL is strong mixing (weaker than independent) and non-identical (though contained in a compact set). We do not know whether there exists a central limit theorem that applies in this context, though we suspect that there should be.

\end{remark}

\subsection{The optimal LLL problem}

The optimal LLL problem (see e.g. \cite{A00}) asks whether LLL with the optimal parameter $\delta = 3/4$ terminates in polynomial time. The following theorem, while crude, shows that this is true for LLL-SP with arbitrarily high probability.

\begin{theorem} \label{thm:optimal}
For any $\eta > 0$ small, LLL-SP with $\delta = 3/4$ terminates after $O_\eta(E)$ steps with probability $1 - \eta$.
\end{theorem}
\begin{proof}
%With the given $\delta$, we have $T = -(1/2)\log3/4$.

Write $\mu$ for the random variable uniformly distributed in $[0, 1/2]$. In case $\delta < 3/4$, the complexity bound of LLL is established with the observation that, with each swap, the energy $E$ decreases by at least $c := \log(\delta + 1/4)^{-1} > 0$, and thus the algorithm must terminate within $E/c$ steps. Similarly, in case $\delta = 3/4$, we try to show that the minimum change of energy $\log(\delta + \mu^2)^{-1}$ is strictly bounded away from zero almost all the time.

(If $I$ was the increment for a given toppling operation, it is easy to show that the energy decreases by $2I$ after such a step.)

Choose a small $\varepsilon > 0$, and let $p = \mathrm{Prob}(\mu \leq 1/2(1-\varepsilon)) = 1- \varepsilon$. Let $d = \log(3/4 + p^2/4)^{-1}$, which is the minimum possible change in energy provided $\mu \leq 1/2(1-\varepsilon)$. Now take $10E/d$ samples $\mu_1, \mu_2, \ldots$ of $\mu$ (there is nothing special about the constant $10$ here). If at least $E/d$ of those samples are less than $1/2(1-\varepsilon)$, LLL-SP would terminate. Proving that this probability is arbitrarily close to $1$ is now a simple exercise with the binomial distribution.

\end{proof}

Observe that the above proof carries over to the case of LLL assuming Conjecture \ref{conj:mu}; the compactness condition on the $\mu_{k(i)}$ distributions allows control on the probability that they are all simultaneously bounded away from $1/2(1-\varepsilon)$.


\begin{thebibliography}{99}

%\bibitem{A04} K. Aardal. Lattice basis reduction in optimization: selected topics. Proceedings of Symposia in Applied Mathematics:
%Trends in Optimization vol.61, pp.1-19, 2004.

\bibitem{A00} A. Akhavi. Worst-case complexity of the optimal LLL algorithm. \emph{LATIN 2000: Theoretical Informatics}, 355-366.

%\bibitem{AHTW16} Y. Aono, T. Hayashi, T. Takagi, Y. Wang. Improved progressive BKZ algorithms and their precise cost estimation by sharp simulator. \emph{Advances in cryptology---EUROCRYPT 2016}, Part I, 789-819, \emph{Lecture Notes in Comput. Sci.}, 9665, Springer, Berlin, 2016.

\bibitem{ACD+18} M. Albrecht, B. Curtis, A. Deo, Alex Davidson, Rachel Player, E. Postlethwaite, Fernando Virdia, Thomas Wunderer. Estimate all the $\{$LWE, NTRU$\}$ schemes! Available at https://github.com/estimate-all-the-lwe-ntru-schemes.

\bibitem{BTW87} P. Bak, C. Tang, and K. Wieselfeld. Self-organized criticality: an explanation of $1/f$ noise. \emph{Phys. Rev. Lett.}, 59:381-384, 1987.

\bibitem{BSW18} S. Bai, D. Stehl\'e, W. Wen. Measuring, simulating and exploiting the head concavity phenomenon in BKZ. \emph{Advances in cryptology --- ASIACRYPT 2018}. Part I, 369-404, Lecture Notes in Comput. Sci., 11272, Springer, Cham, 2018. 

%\bibitem{BP03} J. Beggs and D. Plenz. Neuronal avalanches in neocortical circuits. \emph{J. Neurosci.}, 23:11167-11177, 2003.

\bibitem{BL17} D. Bernstein and T. Lange. Post-quantum cryptography. \emph{Nature}, 549(7671):188-194, 2017.

\bibitem{B95} P. Billingsley. Probability and Measure, 3rd ed. John Wiley \& Sons, 1995.

\bibitem{B05} R. C. Bradley. Basic properties of strong mixing conditions. A survey and some open questions. \emph{Probability Surveys}, Vol. 2, 107-144, 2005.

%\bibitem{BBGP} J. A. Buchmann, D. Butin , F. G\"opfert, A. Petzoldt. Post-Quantum Cryptography: State of the Art. In P. Ryan, D. Naccache, JJ. Quisquate (eds), \emph{The New Codebreakers}, \emph{Lecture Notes in Computer Science}, vol 9100. Springer, Berlin, Heidelberg, 2016.

\bibitem{C88} J. Cardy (ed). Finite-size scaling. Elsevier Science Publishers B.V., 1988.

%\bibitem{C93} H. Cohen, A course in computational algebraic number theory. Springer, 1993.

%\bibitem{CS99} J. Conway and N. Sloane, Sphere Packings, Lattices and Groups. Springer, 1999.

%\bibitem{C97} D. Coppersmith, Small Solutions to Polynomial Equations, and Low Exponent RSA Vulnerabilities. \emph{Journal of Cryptology}, vol.10(4), pp.233-260, 1997.

%\bibitem{Desmos} Desmos Graphing Calculator. Available at https://www.desmos.com/calculator.

\bibitem{D90} D. Dhar. Self-organized critical state of sandpile automaton models. \emph{Phys. Rev. Lett.}, 64(14):1613-1616, 1990.

\bibitem{D99} D. Dhar. The abelian sandpile and related models. \emph{Physica A}, 263(1999) vol. 4, 4-25.

\bibitem{D06} D. Dhar. Theoretical studies of self-organized criticality. \emph{Physica A}, 369(2006) 29-70.

\bibitem{FPLLL} The fpLLL team. fpLLL, a lattice reduction library. Available at https://github.com/fplll/fplll.

\bibitem{GN08} N. Gama, P. Nguyen. Predicting lattice reduction. \emph{Advances in cryptology --- EUROCRYPT 2008}, 31-51, \emph{Lecture Notes in Comput. Sci.}, 4965, Springer, Berlin, 2008.

\bibitem{G09} C. Gentry. A fully homomorphic encryption scheme. PhD thesis, Stanford University, 2009.

\bibitem{GDM16} P. Grassberger, D. Dhar, and P. K. Mohanty. Oslo model, hyperuniformity, and the quenched Edwards-Wilkinson model. \emph{Physical review E} 94, 042314 (2016).

%\bibitem{GR56} B. Gutenberg and C. F. Richter. \emph{Ann. Geophys.}, 9:1, 1956.

\bibitem{HPS11} G. Hanrot, X. Pujol, and D. Stehl\'e. Analyzing blockwise lattice algorithms using dynamical systems. \emph{Advances in cryptology---CRYPTO 2011}, 447-464, Lecture Notes in Comput. Sci., 6841, Springer, Heidelberg, 2011. 

\bibitem{JS98} A. Joux and J. Stern. Lattice reduction: a toolbox for the cryptanalyst. \emph{J. Cryptology} (1998) 11: 161. https://doi.org/10.1007/s001459900042

\bibitem{K15} S. Kim. On the shape of a high-dimensional random lattice. PhD thesis, Stanford University, 2015.

%\bibitem{K18} S. Kim. Random lattice vectors in a set of size $O(n)$. \emph{Int. Math. Res. Not.}, rny060.

\bibitem{KKK17} J. Kim, M. Kim, and S. Kim. LLL via the Lenstra graph. Preprint. Available at https://sites.google.com/view/seungki/

\bibitem{KV17} S. Kim and A. Venkatesh. The behavior of random reduced bases. \emph{Int. Math. Res. Not.}, rnx074.

\bibitem{Kprep} S. Kim and Y. Wang. A stochastic variant of the abelian sandpile model. Preprint. Available at https://sites.google.com/view/seungki/

\bibitem{LLL82} A. K. Lenstra, H. W. Lenstra, Jr., and L. Lov\'asz. Factoring polynomials with rational coefficients. \emph{Math. Ann.}, 261(4):515-534, 1982.

%\bibitem{L08} A. Levina. A mathematical approach to self-organized criticality in neural networks. PhD thesis, University of G\"ottingen, 2008.

\bibitem{L10} L. Levine. What is \ldots a sandpile? \emph{Notices Amer. Math. Soc.} 57 (2010), no. 8, 976-979. 

%\bibitem{MMT98} B. D. Malamud, G. Morein, and D. L. Turcotte. Forest fires: an example of self-organized critical behavior. \emph{Science}, 281(5384):1840-1842, 1998.

\bibitem{NIST} National Institute of Standards and Technology. Post-quantum cryptography, call for proposals. Available at https://csrc.nist.gov/Projects/Post-Quantum-Cryptography/Post-Quantum-Cryptography-Standardization/Call-for-Proposals.

\bibitem{NS16} A. Neumaier and D. Stehl\'e. Faster LLL-type reduction of lattice bases. \emph{Proceedings of the 2016 ACM International Symposium on Symbolic and Algebraic Computation}, 373-380, ACM, New York, 2016.

\bibitem{NS05} P. Nguyen and D. Stehl\'e. Floating-point LLL revisited. \emph{Advances in cryptology --- EUROCRYPT 2005}, volume 3494 of \emph{Lecture Notes in Comput. Sci.}, 215-233. Springer, Berlin, 2005.

\bibitem{NS06} P. Nguyen and D. Stehl\'e. LLL on the average. \emph{Algorithmic number theory}, volume 4076 of \emph{Lecture Notes in Comput. Sci.}, pages 238 - 256. Springer, Berlin, 2006.

\bibitem{NV10} P. Nguyen and B. Vallee (eds). The LLL Algorithm: Survey and Applications. Springer, 2010.

%\bibitem{HP18} H. Hoffman and D. Payton. Optimization by self-organized criticality. \emph{Nature}, 8:2358, 2018.

%\bibitem{NTRU98} J. Hoffstein, J. Pipher, J. H. Silverman, Efficient lattice-based cryptography NTRU: a ring-based public key cryptosystem. Algorithmic Number Theory (ANTS III), Lecture Notes in Computer Science 1423, pp.267-288, 1998.


\bibitem{MV10} M. Madritsch and B. Vall\'e. Modelling the LLL algorithm by sandpiles. \emph{LATIN 2010: theoretical informatics}, 267-281, \emph{Lecture Notes in Comput. Sci.}, 6034, Springer, Berlin, 2010. 

\bibitem{P14} D. Perkinson. Notes for AIMS Cameroon: Abelian Sandpile Model. Available at http://people.reed.edu/~davidp/cameroon

%\bibitem{P96} C. Pomerance. A tale of two sieves. Notices Amer. Math. Soc., vol. 43, 1473-1485, 1996.

\bibitem{R61} V. Rohlin. Exact endomorphisms of a Lebesgue space. \emph{Izvest. Akad. Nauk} 25 (1961), 499-530.

\bibitem{SD09} T. Sadhu and D. Dhar. Steady state of stochastic sandpile models. \emph{J. Stat. Phys.} (2009), 134: 427-441.

\bibitem{S03} P. Schnorr. Lattice reduction by random sampling and birthday methods. \emph{STACS 2003}, 145-156, 
\emph{Lecture Notes in Comput. Sci.}, 2607, Springer, Berlin, 2003. 

\bibitem{SE94} P. Schnorr and M. Euchner. Lattice basis reduction: improved practical algorithms and solving subset sum problems. \emph{Math. Programming} 66 (1994), no. 2, Ser. A, 181-199. 

\bibitem{S45} C.L. Siegel. A mean value theorem in geometry of numbers. Ann. of Math. 46(2) (1945), 340-347.

\bibitem{S10} Damien Stehl\'e. Floating-point LLL: theoretical and practical aspects. In Phong Nguyen and Brigitte Vall\'ee (eds), \emph{The LLL Algorithm, Survey and Applications}, Informaton Security and Cryptography, chapter 5, pages 179-213. Springer-Verlag, Berlin, Heidelberg, 2010.

\bibitem{V16} B. Vall\'ee. Genealogy of lattice reduction: algorithmic description and dynamical analyses. Preprint.

%\bibitem{V10} A. Venkatesh.  A note on sphere packings in high dimension. \emph{Int. Math. Res. Not.} 2013, no. 7, 1628-1642.

\bibitem{DY17} Y. Yu and L. Ducas. Second Order Statistical Behavior of LLL and BKZ. \emph{Selected Areas in Cryptography}, SAC 2017, pp.3-22

\end{thebibliography}
\end{document}